\documentclass[11pt]{amsart}
\usepackage{amscd,amssymb,verbatim}
\theoremstyle{plain}
\newtheorem{Thm}{Theorem}[section]
\newtheorem{Cor}[Thm]{Corollary}
\newtheorem{Lem}[Thm]{Lemma}
\newtheorem{Prop}[Thm]{Proposition}
\newtheorem{Def}[Thm]{Definition}
\newtheorem{remark}[Thm]{Remark}

\numberwithin{equation}{section}

\begin{document}
\title[New examples of \(c_0\)-saturated Banach Spaces]
{New examples of \(c_0\)-saturated  Banach spaces}
\author{I. Gasparis}
\address{Department of Mathematics \\
Aristotle University of Thessaloniki \\
Thessaloniki 54124, Greece.}
\email{ioagaspa@math.auth.gr}
\begin{abstract}
For every \(p \in (1, \infty )\), an isomorphically polyhedral Banach space \(E_p\) 
is constructed which has an unconditional
basis and does not embed isomorphically into a \(C(K)\) space for any countable and compact metric space \(K\).
Moreover, \(E_p\) admits a quotient isomorphic to \(\ell_p\). 
\end{abstract}
\maketitle
\section{Introduction}
Given infinite-dimensional Banach spaces \(E\), \(X\), we call \(X\) {\em \(E\)-saturated}, if every
infinite-dimensional, closed, linear subspace of \(X\) has a closed, linear subspace isomorphic to \(E\). 
It is a well known fact that \(\ell_p\) is \(\ell_p\)-saturated for every \(p \in [1, \infty)\).
A special case of a result of A. Pelczynski and Z. Semadeni \cite{ps}, states that every \(C(K)\) space, where
\(K\) is a countable and compact metric space, is \(c_0\)-saturated. This result was generalized by
V. Fonf \cite{f}, who showed that every infinite-dimensional Banach space whose dual closed unit ball 
contains but countably many extreme points is \(c_0\)-saturated. These spaces are called Lindenstaruss-Phelps 
spaces \cite{lp}. A consequence of Fonf's result is that \(\ell_1\)-preduals are \(c_0\)-saturated.
Fonf \cite{f2} showed that Lindenstrauss-Phelps spaces are {\em isomorphically polyhedral} that is, they
are polyhedral under an equivalent norm. We recall that a Banach space is {\em polyhedral} if the closed unit
ball of every finite-dimensional subspace has finitely many extreme points. He also showed that separable isomorphically
polyhedral spaces are \(c_0\)-saturated and have a separable dual. The converse to this however, is false as 
was shown by D. Leung \cite{l2}.

There are several open problems related to the behavior of isomorphically polyhedral spaces under quotient maps.
In this direction, H. Rosenthal \cite{r2} asked if the dual of a separable isomorphically polyhedral
space is \(\ell_1\)-saturated. We answer this question in the negative by showing the following
\begin{Thm} \label{mainthm}
For every \(p \in (1, \infty)\) there exists an isomorphically polyhedral Banach space \(E_p\)
with an unconditional basis which admits a quotient isomorphic to \(\ell_p\). Moreover, \(E_p\)
is not isomorphic to a subspace of \(C(K)\), for every countable and compact metric space \(K\). 
\end{Thm}
Rosenthal (\cite{r2}, \cite{r3}) has also asked another related question, namely 
if every quotient of an \(\ell_1\)-predual is \(c_0\)-saturated. E. Odell \cite{o}
asked the same question for the special case of \(C(K)\) spaces with \(K\) countable and compact.
Note that if some of the spaces \(E_p\), given in the statement of Theorem \ref{mainthm},
was isomorphic to a subspace of a \(C(K)\) space for some countable and compact metric space \(K\),
then standard duality arguments would yield a quotient of that \(C(K)\) space containing an isomorph
of \(\ell_p\) for some \(p > 1\).

The family of spaces \((E_p)_{p > 1}\) is related to an important example due to D. Alspach \cite{a},
of an \(\ell_1\)-predual which is a quotient of \(C(\omega^\omega)\) and yet not isomorphic
to a subspace of a \(C(K)\) space, for any countable and compact metric space \(K\).
Thus, Alspach's space and the spaces \(E_p\) share common properties, namely they are
isomorphically polyhedral with shrinking bases and do not embed into any \(C(K)\) space
with \(K\) countable and compact. There exist however two major differences between
those spaces. 

The first one is that Alspach's space does not have an unconditional basis, in
contrast to the spaces \(E_p\). In fact, it does not embed into a space with an unconditional basis.
Indeed, it was shown by Rosenthal \cite{r1} that every \(\mathcal{L}_{\infty}\)-subspace
of a space with an unconditional basis must be isomorphic to \(c_0\). Since Alspach's space
is not isomorphic to \(c_0\), it does have the aforementioned property. We also note that
the results of N. Ghoussoub and W. B. Johnson \cite{gj} imply that Alspach's space does not even embed 
into an order continuous Banach lattice.

Another difference between the spaces \(E_p\) and Alspach's space is that none of the
\(E_p\)'s is isomorphic to a quotient of a separable \(\mathcal{L}_{\infty}\)-space. Indeed, 
D. Lewis and C. Stegall (\cite{ls}, \cite{st}) have shown that the dual of an infinite-dimensional,
separable \(\mathcal{L}_{\infty}\)-space is either isomorphic to \(\ell_1\), or to \(C[0,1]^*\),
and so it is a space of cotype \(2\). It turns out however, that
every \(E_p\)
contains uniformly complemented \(\ell_1^n\)'s and therefore it can not be a quotient of 
a separable \(\mathcal{L}_{\infty}\)-space. 

The spaces \((E_p)_{p > 1}\) form a new list of examples of \(c_0\)-saturated Banach spaces which
admit reflexive quotients. The first example of such a space was given by P. Casazza, N. Kalton
and L. Tzafriri \cite{ckt} who showed that a certain Orlicz function space is \(c_0\)-saturated
and admits \(\ell_2\) as a quotient. It was observed in \cite{o} that this space did not posess
an unconditional basis. D. Leung \cite{l} constructed a \(c_0\)-saturated Banach space with an
unconditional basis which has a quotient isomorphic to \(\ell_2\).

Using interpolation methods, S. Argyros and V. Felouzis \cite{af} showed that every reflexive Banach space with
an unconditional basis admits a block subspace which is a quotient of a \(c_0\)-saturated Banach space, as well as
of an \(\ell_p\)-saturated space (\(1 < p < \infty \)). These methods have been recently extended
by S. Argyros and T. Raikoftsalis \cite{ar} to cover all separable reflexive spaces. They show
that every separable reflexive space is a quotient of a \(c_0\)-saturated Banach space with a basis, as
well as of an \(\ell_p\)-saturated space (\(1 < p < \infty \)) with a basis.

We shall next describe how this paper is organized. In section \ref{S1} we give the construction of
a family \((E_p)_{p \in [1, \infty)}\) of spaces having unconditional bases. These spaces are closely
related to the well known example of J. Schreier \cite{s}. We show that every member of this family is isomorphically
polyhedral by applying results of V. Fonf (\cite{f}, \cite{f2}) and J. Elton \cite{e}. 

We prove in section \ref{S2} that \(E_p\) admits a quotient isomorphic to \(\ell_p\) for every
\(p > 1\). We show in particular, that if \((e_n)\) denotes the natural basis of \(E_p\), then there
exist finite subsets \((F_n)\) of \(\mathbb{N}\) with \(\max F_n < \min F_{n+1}\) for all \(n \in \mathbb{N}\)
and \(\lim_n |F_n | = \infty\), so that the convex block basis \((u_n^*)\) of \((e_n^*)\) given by
\(u_n^* = \sum_{i \in F_n} (1/|F_n|) e_i^*\), is equivalent to the usual \(\ell_q\)-basis, where \(q\)
is the exponent conjugate to \(p\). It then follows that the averaging map 
\(Q \colon E_p \to \ell_p\) defined by \(Q(\sum_n a_n e_n ) = (\sum_{i \in F_n} a_i /|F_n|)_n\)
is a bounded, linear surjection.

The proof of the fact that each \(E_p\) is not isomorphic to a subspace of a \(C(K)\) space with
\(K\) countable and compact, is given in section \ref{S3}. By a theorem of S. Mazurkiewicz and
W. Sierpinski \cite{ms},
it will suffice showing that \(E_p\) is not isomorphic to a subspace of 
\(C(\alpha)\) for every countable ordinal \(\alpha\). We recall here that \(C(\alpha)\) 
stands for the space of real valued functions, continuous on the ordinal interval
\([1, \alpha]\) endowed with the order topology. We also recall that the isomorphic classification
of the spaces \(C(\alpha)\) is due to C. Bessaga and A. Pelczynski \cite{bp3}.
 
To accomplish this task, we use a
criterion which gives sufficient conditions for a Banach space with a shrinking unconditional basis 
not to embed in \(C(\alpha)\) for all \(\alpha < \omega_1\). This criterion is inspired by an
argument due to V. Fonf \cite{f} and might be of independent interest. Before giving the precise
statement, we introduce some notation. If \((e_n)\) is a normalized Schauder basis for a Banach space and
\(x\), \(y\) are blocks of \((e_n)\), we write \(y \preceq x \) if \(|e_n^*(y)| \leq |e_n^*(x)|\),
for all \(n \in \mathbb{N}\).
\begin{Thm} \label{T2}
Let \((e_n)\) be a normalized, shrinking, unconditional basis for a Banach space \(X\).
Assume that there exists a semi-normalized block basis \((x_n)\) of \((e_n)\) satisfying the following
property: For every \(M \in [\mathbb{N}]\), every \(\alpha < \omega_1\) and every \(\epsilon > 0\),
there exists a block basis \((u_m)_{m \in M}\) of \((e_n)\) with finitely many non-zero terms so that
\begin{enumerate}
\item \(u_m \preceq x_m, \, \forall \, m \in M\).
\item \(\| \sum_{m \in M} u_m \| = 1\).
\item \(\| \sum_{m \in F} u_m  \| < \epsilon \), for all \(F \subset M\) with \(F \in S_\alpha\),
the Schreier family of order \(\alpha\).
\end{enumerate}
Then, \(X\) is not isomorphic to a subspace of \(C(K)\) for every countable and compact metric space \(K\).
\end{Thm} 
\section{Preliminaries}
Our notation is standard as may be found in \cite{lt}. We shall consider Banach spaces 
over the real field. By a subspace of a Banach space, we always mean a closed, linear subspace. 
If \(X\) is a Banach space then \(B_X\) stands for its closed unit ball. 
A bounded subset \(B\) of the dual \(X^*\) of \(X\) is {\em norming}, if there is some constant
\(\rho > 0\) such that \(\sup_{x^* \in B} |x^*(x)| \geq \rho \|x\|\), for all \(x \in X\).
In case \(B \subset B_{X^*}\) and \(\rho = 1\), \(B\) is called isometrically norming.
\(\mathrm{ ext }(B_{X^*})\) stands for the set of the extreme points of \(B_{X^*}\).
 
\(X\) is said to contain an {\em isomorph} of the Banach space \(Y\) (or, equivalently, that
\(X\) contains \(Y\) isomorphically), if there exists a bounded linear injection from
\(Y\) into \(X\) having closed range. A subspace \(Y\) of \(X\) is said to be {\em complemented} in \(X\), if it
is the range of a linear, idempotent operator on \(X\).

A sequence \((x_n)\) in a Banach space is said to be {\em semi-normalized} if \(\inf_n \|x_n \| > 0 \) and
\(\sup_n \|x_n \| < \infty\). It is called 
a {\em basic} sequence provided it is a Schauder basis for its closed linear
span in \(X\). 
\((x_n)\) is {\em equivalent} to the usual \(\ell_p\) basis for some \(p \in [1, \infty)\)
(resp. \(c_0\) basis), if there exist positive
constants \(\lambda_1 \leq \lambda_2\) such that
\(\lambda_1 (\sum_{i=1}^n |a_i|^p)^{1/p} \leq \|\sum_{i=1}^n a_i x_i \| \leq \lambda_2 
(\sum_{i=1}^n |a_i|^p)^{1/p} \)
(resp. \(\lambda_1 \max_{i \leq n} |a_i| \leq \|\sum_{i=1}^n a_i x_i \| \leq \lambda_2 \max_{i \leq n } |a_i|\))
for every choice of scalars \((a_i)_{i=1}^n\) and all \(n \in \mathbb{N}\). 

If \((e_n)\) denotes the usual \(\ell_p\) basis for some \(p \in [1, \infty)\), (resp. \(c_0\) basis)
and \(k \in \mathbb{N}\), then \(\ell_p^k\) (resp. \(\ell_\infty^k\)) stands for the subspace of \(\ell_p\)
(resp. \(c_0\)) spanned by \((e_i)_{i=1}^k\). A Banach space \(X\) is said to contain 
{\em uniformly complemented} \(\ell_p^n\)'s, for some \(p \in [1, \infty]\), if there exist a constant
\(C > 0\) and linear maps \(T_n \colon \ell_p^n \to X\), \(S_n \colon X \to \ell_p^n\) satisfying
\(S_n T_n x = x \) for all \(x \in \ell_p^n\) and \(\|T_n \| \leq C \), \(\|S_n \| \leq C\),
for all \(n \in \mathbb{N}\).
 
A Schauder basis \((e_n)\) for a Banach space \(X\) is called 
{\em shrinking}, if the sequence \((e_n^*)\)
of functionals biorthogonal to \((e_n)\) is a Schauder basis for \(X^*\).

A basic sequence \((x_n)\) is called {\em suppression} \(1\)-unconditional, if 
for every choice of scalars \((a_i)_{i=1}^n\), and every \(F \subset \{1, \dots , n\}\), we have that
\(\|\sum_{i \in F} a_i x_i \|\) \(\leq \|\sum_{i=1}^n a_i x_i\|\).
Evidently, such a basic sequence is unconditional, that is, every series of the form
\(\sum_n a_n x_n \) converges unconditionally, whenever it converges.

If \((x_n)\) is a basic sequence in some Banach space \(X\), then a sequence \((u_n)\)
of non-zero vectors in \(X\),
is a {\em block basis} of \((x_n)\) if there 
exist a sequence of non-zero scalars \((a_n)\) and a sequence \((F_n)\) of sucessive finite subsets  
of \(\mathbb{N}\) (i.e, \(\max F_n < \min F_{n+1}\) for all \(n \in \mathbb{N}\)), so that
\(u_n = \sum_{i \in F_n} a_i x_i\), for all \(n  \in \mathbb{N}\). We then call \(F_n\) the
{\em support} of \(u_n\) for all \(n \in \mathbb{N}\). Any member of a block basis of
\((x_n)\) will be called a {\em block} of \((x_n)\).

A separable Banach space \(X\) is a \(\mathcal{L}_\infty\)-space (\cite{la}, \cite{lr}), if there exist a constant
\(C > 0\) and an increasing sequence \((X_n)\) of finite-dimensional subspaces of \(X\) with
\(\cup_n X_n \) dense in \(X\) and such that the Banach-Mazur distance between \(X_n\) and
\(\ell_\infty^{\mathrm{dim} X_n }\) does not exceed \(C\) for all \(n \in \mathbb{N}\). 

If \(M\) is an infinite subset of \(\mathbb{N}\), then we let
\([M]\) denote the set of its infinite subsets, while \([M]^{< \infty}\) stands
for the set of all finite subsets of \(M\).

Given finite subsets \(E\), \(F\) of \(\mathbb{N}\), then the notation
\(E < F\) indicates that \(\max E < \min F\). If \(\mu \), \(\nu\) are finitely supported
signed measures on \(\mathbb{N}\), then we write \(\mu < \nu\) if 
\(\mathrm{ supp } \, \mu < \mathrm{ supp } \, \nu\).

A family \(\mathcal{F}\) of finite subsets of \(\mathbb{N}\) is said to be
{\em hereditary}, if \(G \in \mathcal{F}\) whenever \(G \subset F\)
and \(F \in \mathcal{F}\). It is called {\em spreading}, if
\(\{n_1, < \dots, n_k \} \in \mathcal{F}\) whenever
\(\{m_1, < \dots, m_k \} \in \mathcal{F}\) and \(m_i \leq n_i\) for all \(i \leq k\).
\(\mathcal{F} \) is {\em compact}, if it is compact in the topology of pointwise convergence
in \(2^{\mathbb{N}}\).

The Schreier families \(\{S_\alpha: \, \alpha < \omega_1\}\), were introduced in \cite{aa}.
\[S_1 = \{F \in [\mathbb{N}]^{< \infty}: \, |F| \leq \min F \} \cup \{\emptyset\}.\]
The higher ordinal Schreier families are then defined by transfinite induction.
We do not give the precise definition, since it will not be needed in the sequel.
All we need to know, is that 
they are hereditary, spreading and compact, and that they exhaust
the complexity of the countable and compact metric spaces.
More precisely, it is shown in \cite{aa} that
\(S_\alpha\) is homeomorphic to the ordinal interval
\([1, \omega^{\omega^\alpha}]\), for all \(\alpha < \omega_1\). 
\section{The construction of \(E_p\)} \label{S1}
A lot of examples of Banach spaces with unconditional bases, such as Schreier's \cite{s}
or Tsirelson's \cite{t} spaces, are constructed through the following general method:  
Given \(n \in \mathbb{N}\), let \(e_n^*\) denote the point mass measure at \(n\).
In order to simplify our notation, if 
\(\mu\) is a finitely supported signed measure on \(\mathbb{N}\) and \(n \in \mathbb{N}\), we
write \(\mu(n)\) instead of \(\mu(\{n\})\). Suppose that \(\mathcal{M}\) is a collection of 
finitely supported signed measures on \(\mathbb{N}\) satisfying the following properties:

(a): \(e_n^* \in \mathcal{M}\), for all \(n \in \mathbb{N}\).

(b): \(|\mu(n)| \leq 1\), for all \(\mu \in \mathcal{M}\) and \(n \in \mathbb{N}\).

(c): If \( \mu \in \mathcal{M}\) and \(I \subset \mathbb{N}\), then \(\mu | I \in \mathcal{M}\).

Then, we can define a norm \(\| \cdot \|_{\mathcal{M}}\) on \(c_{00}\) in the following manner:
\[\|x \|_{\mathcal{M}} = \sup \biggl \{ \biggl | \sum_n x(n) \mu(n) \biggr | : \, \mu \in \mathcal{M} \biggr \},
\, \forall \, x \in c_{00}.\]
Let \(X_{\mathcal{M}}\) denote the completion of \((c_{00}, \| \cdot \|_{\mathcal{M}})\).
The conditions imposed on \(\mathcal{M}\) ensure that \((X_{\mathcal{M}}, \| \cdot \|_{\mathcal{M}})\)   
is a Banach space and that the natural basis \((e_n)\) of \(c_{00}\) is a normalized, suppression \(1\)-unconditional
Schauder basis for \(X_{\mathcal{M}}\). Moreover, the elements of \(\mathcal{M}\) act naturally
as bounded linear functionals
on \(X_{\mathcal{M}}\) and in fact, \(B_{X_{\mathcal{M}}^*}\) is the \(w^*\)-closed convex hull in  
\(X_{\mathcal{M}}^*\) of \(\mathcal{M} \cup - \mathcal{M}\).
In particular, \(\mathcal{M}\) isometrically norms \(X_{\mathcal{M}}\) and \((e_n^*)\) becomes the
sequence of functionals biorthogonal to \((e_n)\). We also have that \((e_n^*)\) is normalized.

For our purposes, given \(p \in [1, \infty)\), we shall choose a particular set \(\mathcal{M}_p\) of
non-negative measures on \(\mathbb{N}\), subject to conditions (a)-(c), and let
\(E_p = (X_{\mathcal{M}_p}, \| \cdot \|_{\mathcal{M}_p})\). To do so, we first fix a sequence
\((F_n)\) of successive finite subsets of \(\mathbb{N}\). This choice of \((F_n)\) does
affect the definition of \(E_p\). However, the isomorphic properties of \(E_p\) we are    
interested in, are independent of this choice as long as \((F_n)\) satisfies \(\lim_n |F_n| = \infty\).
More spesifically, \(E_p\) will be isomorphically polyhedral independently of the choice of \((F_n)\).
\(E_p\) will not be isomorphic to a subspace of a \(C(K)\) space for any countable and compact metric space \(K\), 
if \(\lim_n |F_n| = \infty\). The same growth condition on \((F_n)\) will give us that \(\ell_p\) is isomorphic to a quotient
of \(E_p\) when \(p > 1\).
\begin{Def}
\begin{enumerate}
\item A finitely supported measure \(\mu\) on \(\mathbb{N}\) is a \(p\)-measure, if there exists a sequence
\((G_n)\) of finite subsets of \(\mathbb{N}\)
with \(G_n\) an initial segment of \(F_n\) (we allow \(G_n = \emptyset\)), for all \(n \in \mathbb{N}\)
so that
\[\mu = \sum_n \frac{|G_n|^p}{|F_n|^p} \, e_{\max G_n}^*, \, \text{ and } \, \sum_n \frac{|G_n|^p}{|F_n|^p} \leq 1\]
(in case \(G_n = \emptyset\) for some \(n \in \mathbb{N}\), we interpret the corresponding summand as the zero measure).
\item Let \(\mu_1 < \dots < \mu_k \) be a finite sequence of non-zero \(p\)-measures. For each \(i \leq k\), let
\(m_i\) denote the smallest \(n \in \mathbb{N}\) satisfying \(F_n \cap \mathrm{ supp } \, \mu_i \, \ne \emptyset\).
The sequence \((\mu_i)_{i=1}^k \) is called admissible, if \(\{ \min F_{m_i} : \, i \leq k \} \in S_1\).
\end{enumerate} 
\end{Def}
We now define
\begin{align}
\mathcal{M}_p &= \biggl \{ \sum_{i=1}^k \mu_i : \, k \in \mathbb{N}, \, \mu_1 < \dots < \mu_k  
\text{ is an admissible sequence of} \notag \\
&p-\text{measures } \biggr \} \cup \{e_n^* : \, n \in \mathbb{N}\}.\notag
\end{align}
It is clear from the definition that \(\mathcal{M}_p\) satisfies conditions (a) and (b).
The spreading property of \(S_1\) gives us that (c) is also satisfied. Hence, \((e_n)\) is a
suppression \(1\)-unconditional basis for \(E_p\).

We are going to show that \(E_p\) is \(c_0\)-saturated. This will be accomplished through 
our next proposition, which provides some natural conditions imposed on a 
bounded norming subset of the dual of a Banach space with a basis, that 
ensure that the space is \(c_0\)-saturated. The proof of this proposition relies on Elton's theorem
\cite{e} and a compactness argument used in the proof of the reflexivity of mixed
Tsirelson spaces \cite{ad}.
\begin{Prop} \label{P1}
Let \(X\) be a Banach space with a normalized Schauder basis \((e_n)\). Let
\((e_n^*)\) denote the sequence of functionals biorthogonal to \((e_n)\).
Assume there is a bounded norming subset \(B\) of \(X^*\) with the following property:
There exists a compact family \(\mathcal{F}\) of finite subsets
of \(\mathbb{N}\) such that for every \(b^* \in B\) there exist \(F \in \mathcal{F}\)
and a finite sequence \((b_k^*)_{k \in F}\) of finitely supported absolutely 
sub-convex combinations of \((e_n^*)\) so that
\(b^* = \sum_{k \in F} b_k^*\) and
\(\min \mathrm{ supp } \, b_k^* \geq k \) for all \(k \in F\).
Then, \(\sum_n |b^*(e_n)| < \infty\), for all \(b^* \in \overline{B}^{w^*}\) and 
\(X\) is \(c_0\)-saturated.
\end{Prop} 
\begin{proof}
Let \(K\) denote the \(w^*\)-closure of \(B\) in \(X^*\). \(K\) is a \(w^*\)-compact subset
of \(X^*\) which norms \(X\) and thus \(X\) embeds isomorphically in \(C(K)\).
In view of Elton's theorem \cite{e}, it will suffice showing that 
\((x^*(e_n))\) is an absolutely summable sequence for every 
\(x^* \in K \). To this end, suppose that \(B \subset d B_{X^*}\), for some \(d > 0\),
and let \((b_n^*)\) be a sequence in \(B\) 
\(w^*\)-converging to some \(x^* \in K \). Our assumptions allow us find, for
every \(n \in \mathbb{N}\), some \(F_n \in \mathcal{F}\) and a 
finite sequence \((b_{nk}^*)_{k \in F_n}\) of finitely supported absolutely sub-convex
combinations of \((e_n^*)\) so that
\[b_n^* = \sum_{k \in F_n } b_{nk}^*, \text{ and }
\min \mathrm{ supp } \, b_{nk}^* \geq k, \, \forall \, k \in F_n.\]
We may assume, without loss of generality by passing to subsequences if necessary thanks 
to the compactness of \(\mathcal{F}\), that \((F_n)\) converges pointwise to some \(F \in \mathcal{F}\).
Moreover, we may assume that for every \(n \in \mathbb{N}\), \(F_n = F \cup G_n\) with \(G_n\) disjoint
from \(F\) and \(\lim_n \min G_n = \infty\) (this is the case when \((F_n)\) has no constant 
subsequence. Otherwise, the rest of the argument becomes simpler and it concludes as in the last paragraph
of the present proof). It follows from this that
\(w^*- \lim_n \sum_{k \in G_n} b_{nk}^* = 0\).

Indeed, write \(u_n^* = \sum_{k \in G_n } b_{nk}^*\) and 
\(b_n^* = \sum_{k \in F} b_{nk}^* + u_n^*\), for all \(n \in \mathbb{N}\).
If \( M = \sup_n \|e_n^*\|\), then \(\|u_n^*\| \leq M |F| + d\), for all \(n \in \mathbb{N}\),
as \(b_n^* \in B \subset d B_{X^*}\) and \(\|b_{nk}^*\| \leq M\) for all \( k \in F_n\) and
\(n \in \mathbb{N}\). Hence, \((u_n^*)\) is a uniformly bounded in norm sequence of finite linear 
combinations of \((e_n^*)\) with \(\lim_n \min \mathrm{ supp } \, u_n^* = \lim_n \min G_n = \infty\) (assuming,
as we clearly may, that \(u_n^* \ne 0\) for all \(n \in \mathbb{N}\)) and so
\(w^*- \lim_n u_n^* = 0\), as required.  

We conclude that \(x^* = w^*- \lim_n \sum_{k \in F} b_{nk}^*\).
Once again, after passing to subsequences and diagonalizing, we can assume that
\(w^*- \lim_n b_{nk} = b_k^* \in X^*\), for all \(k \in F\) and thus,
\(x^* = \sum_{k \in F} b_k^*\). Fix some \(l \in F \).
Since \((b_{nl}^*)_n \) is a sequence of absolutely sub-convex combinations of \((e_n^*)\),
we must have that \(\sum_n |b_l^*(e_n)| \leq 1\). Therefore,
\[\sum_n |x^*(e_n)| \leq \sum_{k \in F} \sum_n |b_k^*(e_n)| \leq |F|\]
and so \((x^*(e_n))\) is absolutely summable. This completes the proof.
\end{proof}
\begin{remark}
In the applications of Proposition \ref{P1}, the sequence \((b_k^*)_{k \in F}\) 
will consist of disjointly supported absolutely sub-convex combinations of \((e_n^*)\)
with \(\{\min \mathrm{ supp } \, b_k^* : \, k \in F \} \in \mathcal{F}\).
\end{remark}
We shall actually show that \(E_p\) is isomorphically polyhedral. This is a consequence of
the next proposition, which is most likely known although we have not been able to find an
exact reference.
\begin{Prop} \label{P3}
Let \(X\) be a Banach space with a normalized Schauder basis \((e_n)\).
Assume that there exists a \(w^*\)-compact, norming subset \(B\) of \(X^*\) 
such that
\[\sum_n |b^* (e_n) | < \infty, \, \forall \, b^* \in B.\]
Then, \(X\) is isomorphic to a Lindenstrauss-Phelps space and therefore, it
is isomorphically polyhedral.
\end{Prop}
\begin{proof}
We first observe that \(X\) naturally embeds into \(C(B)\), as \(B\) is norming.
Since \((b^*(e_n))\) is absolutely summable for every \(b^* \in B\) and \((e_n)\) is
normalized, we obtain that \(\lim_n b^*(u_n) = 0\) for every normalized block basis
\((u_n)\) of \((e_n)\), and every \(b^* \in B\). It follows now that \((u_n)\) is weakly null
and thus, \((e_n)\) is shrinking. We next define
\[\mathcal{N} = \{ \sigma_1 b^* | [1,n] + \sigma_2 e_k^* : \, b^* \in B,  
\, \sigma_i \in \{-1, 0,1\} \, (i=1,2), \, n < k \text{ in } \mathbb{N}\}.\]
This is a bounded subset of \(X^*\). 
We can now define a norm \(\| \cdot \|_{\mathcal{N}}\) on \(X\) by 
\[\|x \|_{\mathcal{N}} = \sup_{\nu \in \mathcal{N}} |\nu(x)|,
\, \forall \, x \in X.\]
Let \(Y = (X, \|\cdot\|_{\mathcal{N}})\).
Evidently, as \(B\) is norming, \(\|\cdot \|_{\mathcal{N}}\) is an equivalent norm
on \(X\).

We claim that \(\mathrm{ext}(B_{Y^*}) \subset \mathcal{N}\).
Let us assume momentarily that this claim is proved. Then, clearly, \(\mathcal{N}\) can be covered by
a countable union of norm-compact subsets of \(Y^*\), and so the same holds true for 
\(\mathrm{ext}(B_{Y^*}) \). A result of Elton \cite{e} now yields that \(Y\) is isomorphic to
a Lindenstrauss-Phelps space and hence, \(X\) is isomorphically polyhedral by the results of \cite{f2}.

To prove the claim, we employ Milman's theorem, as \(\mathcal{N}\)
isometrically norms \(Y\), and obtain that
\(\mathrm{ext}(B_{Y^*}) \subset \overline{\mathcal{N}}^{w^*}\).
We are going to show that if \(\nu \in \overline{\mathcal{N}}^{w^*}\) is infinitely 
supported with respect to \((e_n^*)\), then \(\nu\) is not an extreme point of
\(B_{Y^*}\). This clearly implies the validity of the claim.
Now, since \(\nu\) is infinitely supported, we may write \(\nu = \sigma b^*\)
for some infinitely supported \(b^* \in B\) and \(\sigma \in \{-1,1\}\).
Since \((b^*(e_n))\) is absolutely summable, we can find \(n_0 \in \mathbb{N}\)
such that \(\sum_{n > n_0 } |b^*(e_n) | < 1\). We next write
\[\nu = \sigma b^* |[1, n_0] + \sum_{n > n_0} \sigma b^*(e_n) e_n^* \]
and hence, for a suitable choice of signs \((\sigma_n)_{n > n_0}\), we have that
\[\nu = \sum_{n > n_0} |b^*(e_n)| (\sigma b^* |[1, n_0] + \sigma_n e_n^*)
+ (1 - \sum_{n > n_0} |b^*(e_n)|) \sigma b^* | [1,n_0].\]
Since \(b^*\) is infinitely supported, we have managed to express \(\nu\) as a convex combination of distinct elements 
of \(\mathcal{N} \subset B_{Y^*}\) and so it can not be an extreme point of \(B_{Y^*}\). 
\end{proof}
\begin{remark}
A special case of a result due to G. Androulakis \cite{an} implies that if the \(w^*\)-compact norming set \(B\)
in the statement of Proposition \ref{P3} is replaced by a boundary set for \(X\), i.e., a subset of
the unit sphere of \(X^*\) on which every non-zero vector of \(X\) attains its norm, then some
subsequence of \((e_n)\) spanns an isomorphically polyhedral subspace of \(X\).
\end{remark}  
\begin{Cor} \label{C1}
\(E_p\) is isomorphically polyhedral and \((e_n)\) is an unconditional, shrinking normalized basis for \(E_p\).
\end{Cor}
\begin{proof}
It follows directly from the definition of \(\mathcal{M}_p\), that \((e_n)\) is a normalized,
suppression \(1\)-unconditional basis for \(E_p\).
Let \( \mu \in \mathcal{M}_p\). We first verify that the conditions given in Proposition \ref{P1}
are fulfilled by \(\mu\) with \(\mathcal{F} = S_1\). Indeed, this is obvious when \(\mu = e_n^*\)
for some \(n \in \mathbb{N}\). Otherwise, \(\mu = \sum_{i=1}^k \mu_i \) for some \(k \in \mathbb{N}\) 
and \(p\)-measures \(\mu_1 < \dots < \mu_k\) with \(\{\min \mathrm{ supp } \, \mu_i : \, i \leq k \}
\in S_1\). Since every \(p\)-measure is a finitely supported sub-convex combination of \((e_n^*)\),
we deduce from Proposition \ref{P1} that \((b^*(e_n))\) is absolutely summable for every
\(b^* \in \overline{\mathcal{M}_p}^{w^*}\), and that \(E_p\) is \(c_0\)-saturated.
We have already observed in the proof of Proposition \ref{P3}
that the summability property of \(\overline{\mathcal{M}_p}^{w^*}\), 
implies that \((e_n)\) is shrinking. This fact of course could be deduced from
James's characterization of shrinking unconditional bases, as \(E_p\) is \(c_0\)-saturated and hence 
it can not contain an isomorph of \(\ell_1\).
Finally, an appeal to Proposition \ref{P3} gives us that
\(E_p\) is isomorphically polyhedral. 
\end{proof}  
\begin{remark}
J. Lindenstrauss and R. Phelps \cite{lp} asked
whether the dual of a Lindenstrauss-Phelps space has the Schur property.
This question was answered in the negative by E. Tokarev \cite{to}.
The spaces \(E_p\) for \(p > 1\), also provide counterexamples to the preceding question, since it will be shown
in the next section that they admit quotients isomorphic to \(\ell_p\). 
A simple example of a Lindenstrauss-Phelps space whose
dual fails the Schur property, is Schreier's space \(\mathcal{S}\). Indeed, \(\mathcal{S}\), as presented
in \cite{o}, is
isomorphic to a subspace of \(C(\omega^\omega)\) and hence it is isomorphic to a Lindenstrauss-Phelps space.
Letting \((e_n)\) denote the natural normalized Schauder basis of \(\mathcal{S}\), then \((e_n)\)
is shrinking and \((e_n^*)\) is a normalized weakly null sequence in \(\mathcal{S}^*\). It follows now that
\(\mathcal{S}^*\) fails the Schur property. It can be shown that \(\mathcal{S}^*\) is \(\ell_1\)-saturated.
\end{remark}
The next lemma describes a simple method for selecting subsequences of \((e_n)\), equivalent to
the \(c_0\)-basis.
\begin{Lem} \label{L1}
\begin{enumerate}
\item Let \((G_n)\) be a sequence of finite subsets of \(\mathbb{N}\) with \(G_n\) an initial
segment of \(F_n\) for all \(n \in \mathbb{N}\), such that \(\sum_n (|G_n|/|F_n|)^p \leq 1\).
Then, \(\sum_n \mu(G_n) \leq 1\), for all \(\mu \in \mathcal{M}_p\).
\item Let \((G_n)_{n \in I}\) be a finite sequence of finite subsets of \(\mathbb{N}\) with \(G_n\) an initial
segment of \(F_n\) for all \(n \in I \subset \mathbb{N}\), such that \(\sum_n (|G_n|/|F_n|)^p \leq 1\).
Then, \(\| \sum_{n \in I} \sum_{k \in G_n} e_k \| \leq 1\).
\end{enumerate}
\end{Lem}
\begin{proof}
We need only prove \((1)\) as \((2)\) follows immediately from \((1)\).
Let \(\mu \in \mathcal{M}_p\). In case \(\mu = e_n^*\)
for some \(n \in \mathbb{N}\), then the assertion holds trivially.
Every other element of \(\mathcal{M}_p\) admits a representation
of the form \(\mu = \sum_{n=1}^d (|H_n|/|F_n|)^p e_{\max H_n}^* \) for some \(d \in \mathbb{N}\) 
and finite sets \((H_n)_{n=1}^d \) with \(H_n\) an initial segment of \(F_n\), for all \(n \leq d \). 

Now let \(n \in \mathbb{N}\) with \(n \leq d \). It is clear that either \(H_n\) is an initial segment
of \(G_n \), or \(\max H_n > \max G_n\). If the former, then 
\(\mu(G_n) = (|H_n|/|F_n|)^p \leq (|G_n|/ |F_n| )^p \).
If the latter, then \(\mu(G_n) = 0\).
In any case, we have that \(\mu(G_n) \) \(\leq (|G_n|/ |F_n| )^p \), for all \(n \leq d\).
Of course, \(\mu(G_n) = 0\), for all \(n > d\). Hence,
\[\sum_n \mu(G_n) \leq \sum_n (|G_n|/|F_n|)^p  \leq 1, \] 
as required.
\end{proof}
\begin{remark}
It follows directly from the definition of the norming set \(\mathcal{M}_p\), that for every
\(p \geq 1\), the subsequence \((e_{\max F_n})\) of the basis \((e_n)\) of \(E_p\) generates
a spreading model equivalent to the usual \(\ell_1\) basis. In fact, it is equivalent to the
corresponding subsequence of the natural basis of Schreier's space. Consequently, \(E_p\)
contains uniformly complemented \(\ell_1^n\)'s.
\end{remark}
\section{ \(E_p\), \(p > 1\), admits a quotient isomorphic to \(\ell_p\)} \label{S2}
In this section we assume that \(p > 1\) and that \(\sum_n 1 / |F_n| < 1 \). Our objective
is to show that \(E_p\) admits a quotient isomorphic to \(\ell_p\). Let \(q > 1\) be the 
exponent conjugate to \(p\). We first introduce some notation.

{\bf Notation}. Fix scalars \(\alpha > 1 \) and \(\beta > 0\) satisfying
\(\alpha + \frac{\beta}{q-1} < p\). Let \(\gamma > 0\) satisfy
\(\alpha + \frac{\beta}{q-1} + \gamma = p\). 

We also fix a scalar \(\theta \in (0,1)\) and set \(\epsilon_n = \theta^{\alpha^n}\), for
all \(n \in \mathbb{N}\). Note that \(\epsilon_{n+1} = \epsilon_n^{\alpha}\), for all \(n \in \mathbb{N}\).
It is clear that \(\sum_n n \epsilon_n^{\lambda} < \infty\), for
all \(\lambda > 0\).

We shall also require that \(\min F_n > 1 + 2/\epsilon_n \), for all \(n \in \mathbb{N}\).
The main result of this section is the following
\begin{Thm} \label{T3}
Let \(u_n^* = \sum_{i \in F_n} (1/|F_i|) e_i^* \), for all \(n \in \mathbb{N}\).
Then \((u_n^*)\) is equivalent to the usual basis of \(\ell_q\).
\end{Thm}
In the next series of lemmas, we show that \((u_n^*)\) satisfies upper and lower 
\(\ell_q\) estimates. The lower estimates will be straightforward, while the upper ones require
some technical calculations. In the sequel, \((u_n^*)\) is the sequence given in Theorem \ref{T3}.
\begin{Lem} \label{L2}
\((u_n^*)\) dominates the usual \(\ell_q\) basis.
\end{Lem}
\begin{proof}
Note first that \((u_n^*)\) is a normalized (convex) block basis of \((e_n^*)\) in \(E_p^*\).
This is so since \(\|\sum_{i \in F_n } e_i \| = 1\), for all \(n \in \mathbb{N}\) (observe that the support of every
member of \(\mathcal{M}_p\) meets each \(F_n\) in at most one point).   
We thus obtain that \((u_n^*)\) is normalized and suppression \(1\)-unconditional.

Now let \(n \in \mathbb{N}\) and \((a_i)_{i=1}^n\) be scalars in \([0,1]\) with \(\sum_{i=1}^n a_i^q \leq 1\).
For each \(i \leq n\) choose an initial segment \(G_i\) of \(F_i\) so that
\[|G_i|/|F_i| \leq a_i^{q-1} < (|G_i| / |F_i|) + (1/|F_i|).\]
This choice ensures that 
\[\sum_{i=1}^n (|G_i|/|F_i|)^p \leq \sum_{i=1}^n a_i^{p(q-1)} = \sum_{i=1}^n a_i^q \leq 1.\]
Let \(u = \sum_{i=1}^n \sum_{k \in G_i} e_k \). We deduce from Lemma \ref{L1}, that \(\|u\| \leq 1\).
It follows now that 
\begin{align}
\biggl \|\sum_{i=1}^n a_i u_i^*  \biggr \| &\geq \sum_{i=1}^n a_i u_i^*(u) 
= \sum_{i=1}^n a_i \sum_{k \in F_i} (1/|F_i|) e_k^*(u) = \sum_{i=1}^n a_i (|G_i|/|F_i|) \notag \\
&\geq \sum_{i=1}^n a_i (a_i^{q-1} - (1/|F_i|)) \geq \sum_{i=1}^n a_i^q - \sum_{i=1}^n (1/|F_i|). \notag
\end{align}
Next suppose that \(n \in \mathbb{N}\) and \((a_i)_{i=1}^n\) is a scalar sequence
satisfying \(\sum_{i=1}^n |a_i|^q\) \(=1\). Let \(I^{+} = \{i \leq n : \, a_i \geq 0 \}\)
and \(I^{-} = I \setminus I^{+}\). Our preceding work yields that
\[ \biggl \|\sum_{i \in I^j} a_i u_i^*  \biggr \| \geq  \sum_{i \in I^j} |a_i|^q -
\sum_{i \in I^j} (1/|F_i|), \, \forall \, j \in \{+, -\}.\]
We deduce now from the above and the fact that \((u_n^*)\) is suppression \(1\)-unconditional, that
\[2 \biggl \|\sum_{i=1}^n a_i u_i^*  \biggr \| \geq 1 - \sum_{i=1}^\infty (1/|F_i|) > 0. \]
Therefore, letting \(A = (1/2)(1 - \sum_{i=1}^\infty (1/|F_i|)) > 0\), we obtain that
\[\biggl \|\sum_{i=1}^n a_i u_i^* \biggr \| \geq A 
\biggl ( \sum_{i=1}^n |a_i|^q \biggr )^{1/q}\] 
for every \(n \in \mathbb{N}\) 
and all choices of scalars \((a_i)_{i=1}^n \subset \mathbb{R}\). The proof of the lemma is now complete.
\end{proof}  
\begin{Lem} \label{L3}
Let \(u = \sum_n a_n e_n \) be a finitely supported vector in \(E_p\), with \(\|u \| \leq 1\)
and \(a_n \geq 0\), for all \(n \in \mathbb{N}\). Suppose that there exist \(n \in \mathbb{N}\),
a finite set \(I \subset \mathbb{N}\) and integers \((j_i)_{i \in I}\) 
with \(j_i \in F_i \) for all \(i \in I\), so that \(a_{j_i} \geq \epsilon_n \), for all \(i \in I\).
Then, letting \(G_i = [\min F_i , j_i] \cap F_i\) for all \(i \in I\), the following estimate holds:
\[\sum_{i \in I} a_{j_i} (|G_i|/|F_i|)^p \leq (n+1) \|u\|.\]
\end{Lem}
\begin{proof}
Set \(I_1 = \{i \in I : \, i \geq n \}\). Then, set \(I_2 = \{i \in I_1 : \, (|G_i|/|F_i|)^p \geq 1/2 \}\)
and \(I_3 = I_1 \setminus I_2 \). It is clear that
\[\sum_{i \in I \setminus I_1} a_{j_i} (|G_i|/|F_i|)^p \leq (n-1) \max_{i \in I} a_{j_i} \leq (n-1) \|u\|.\]
The assertion of the lemma will follow once we show that for each \(s \in \{2,3\}\) there exists
\(\tau_s \in \mathcal{M}_p\) such that
\[\sum_{i \in I_s} a_{j_i} (|G_i|/|F_i|)^p = \tau_s(u).\]
To this end, we first claim that \(|I_2| \leq [2/\epsilon_n]\).

Indeed, if the claim were false, \(I_2\) would contain at least \(m = 1 + [2/\epsilon_n]\) elements.
So let
\(i_1 = \min I_2\) and 
choose \(i_1 < i_2 < \dots < i_m \) in \(I_2\). Consider the \(p\)-measures
\(\nu_k = (|G_{i_k}|/|F_{i_k}|)^p e_{j_{i_k}}^*\), for all \(k \leq m\) (observe that by definition,
\(G_i\) is an initial segment of \(F_i\) for all \(i \in I\)).
We know that \(i_1 \geq n\) and so \(\min F_{i_1} \geq \min F_n > 2/\epsilon_n\), by our initial
assumptions on \((F_n)\).
This in turn implies that \(\min F_{i_1} \geq m\) and thus,
\(\{\min F_{i_k} : \, k \leq m \} \in S_1\). It follows now that
\((\nu_k)_{k=1}^m\) is admissible and so \(\nu = \sum_{k=1}^m \nu_k\)
belongs to \(\mathcal{M}_p\). Therefore,
\[1 \geq \|u\| \geq \nu(u) = \sum_{k=1}^m a_{j_{i_k}}(|G_{i_k}|/|F_{i_k}|)^p \geq m (\epsilon_n /2), \]
using the hypothesis that \(a_{j_{i_k}} \geq \epsilon_n \) and the fact that \(i_k \in I_2\) for all \(k \leq m\).
We infer from the above, that \(m \leq 2/\epsilon_n\), contradicting the definition of \(m\).
This contradiction proves our claim.

Setting now \(\tau_2 = \sum_{i \in I_2} (|G_i|/|F_i|)^p e_{j_i}^*\) we obtain, via the claim, 
that \(\{\min F_i : \, i \in I_2 \} \in S_1\), as \(\min F_{i_1} \geq \min F_n > 2/\epsilon_n\).
Hence, \(\tau_2 \in \mathcal{M}_p\). It is also clear that
\(\sum_{i \in I_2} a_{j_i} (|G_i|/|F_i|)^p = \tau_2(u) \).

We next pass to the search for \(\tau_3\). We first choose a (necessarily non-empty) initial segment
\(J_1\) of \(I_3\) which is maximal with respect to the condition
\(\sum_{i \in J_1} (|G_i|/|F_i|)^p \leq 1\). In case \(J_1 = I_3\), we set
\(\tau_3 = \sum_{i \in I_3} (|G_i|/|F_i|)^p e_{j_i}^* \). Then, \(\tau_3\) is a \(p\)-measure
and so it belongs to \(\mathcal{M}_p\). It is clear that, in this case, 
\(\sum_{i \in I_3} a_{j_i} (|G_i|/|F_i|)^p = \tau_3(u) \).

If \(J_1\) is a proper initial segment of \(I_3\), then, by maximality, we must have
\[1/2 < \sum_{i \in J_1} (|G_i|/|F_i|)^p \leq 1, \]
as \( (|G_i|/|F_i|)^p < 1/2\), for every \(i \in I_3\).
We now set \(\mu_1 = \sum_{i \in J_1} (|G_i|/|F_i|)^p e_{j_i}^* \). 
This is a \(p\)-measure satisfying \(\mu_1(u) \geq \epsilon_n /2\) because
\(a_{j_i} \geq \epsilon_n \), for all \(i \in I\), and 
\(\sum_{i \in J_1} (|G_i|/|F_i|)^p > 1/2\).

We repeat the same process to \(I_3 \setminus J_1\) and obtain an initial segment \(J_2\)
of \(I_3 \setminus J_1\), and a \(p\)-measure \(\mu_2 = \sum_{i \in J_2} (|G_i|/|F_i|)^p e_{j_i}^* \)
so that either \(J_1 \cup J_2 = I_3\), or, \(J_2\) is a proper initial segment of
\(I_3 \setminus J_1\) satisfying \(\mu_2(u) \geq \epsilon_n /2\). 
If the former, the process stops. If the latter, the process continues.
Because \(I_3\) is finite, this process will terminate after a finite number of steps, say \(k\).
We shall then have produced successive subintervals \(J_1 < \dots < J_k \) of \(I_3\) 
with \(I_3 = \cup_{r=1}^k J_r\), and
\(p\)-measures \(\mu_1 < \dots < \mu_k \) with
\(\mu_r = \sum_{i \in J_r} (|G_i|/|F_i|)^p e_{j_i}^* \), for all \(r \leq k\).
Moreover, \(\mu_r(u) \geq \epsilon_n /2\), for all \(r < k\).

We claim that \(k \leq m = 1 + [2/\epsilon_n ]\). 
Indeed, assuming \(m < k\), we have by the choice of \(k\), that
\(\mu_r(u) \geq \epsilon_n /2\), for all \(r \leq m\).
But also, \(m < \min F_n \leq \min F_{\min J_1}\), since \(J_1 \subset I_1 \), and thus,
\((\mu_r)_{r=1}^m\) is admissible. Therefore,
\[1 \geq \|u\| \geq \sum_{r=1}^m \mu_r (u) \geq m \epsilon_n /2,\]
whence \(m \leq 2/\epsilon_n\) which is a contradiction. We finally set
\(\tau_3 = \sum_{r=1}^k \mu_r \).
Since \(k \leq m \leq \min F_{\min J_1}\), we have that \((\mu_r)_{r=1}^k \)
is admissible and so \(\tau_3 \in \mathcal{M}_p\). But also, \(I_3 = \cup_{r=1}^k J_r\) and so 
\(\sum_{i \in I_3} a_{j_i} (|G_i|/|F_i|)^p = \tau_3(u) \), completing the proof of
the lemma.
\end{proof}
\begin{Lem} \label{L4}
Let \(u = \sum_n a_n e_n \) be a finitely supported vector in \(E_p\), with \(\|u \| \leq 1\)
and \(a_n \geq 0\), for all \(n \in \mathbb{N}\). Let \(I\) be a finite subset of \(\mathbb{N}\).
For each \(i \in I\), let \(G_i^0\) denote the smallest initial segment of
\(F_i\) that contains \(\{j \in F_i : \, a_j \geq \epsilon_1 \}\) as a subset, while for \(n \in \mathbb{N}\),
let \(G_i^n\) denote the smallest initial segment of
\(F_i\) that contains \(\{j \in F_i : \, \epsilon_{n+1} \leq a_j < \epsilon_n \}\) as a subset.
Also, let \(j_i^n = \max G_i^n \) for every \(n \in \mathbb{N} \cup \{0\}\) with \(G_i^n \ne \emptyset\).
Then, given a finite sequence of non-negative scalars \((\rho_i)_{i \in I}\), the following estimates hold:
\begin{align}
\sum_{i \in I} \rho_i \sum_{j \in F_i : \, a_j \geq \epsilon_1} a_j /|F_i| &\leq
(1/\epsilon_1) \sum_{i \in I} a_{j_i^0} (|G_i^0|/|F_i|)^p + \sum_{i \in I} \rho_i^q  \label{E1} \\
\sum_{i \in I} \rho_i \sum_{j \in F_i : \, a_j < \epsilon_n} a_j /|F_i| &\leq 
\sum_{i \in I} \rho_i \sum_{j \in F_i : \, a_j < \epsilon_{n + 1}} a_j /|F_i| \, +  \label{E2} \\
&+ \epsilon_n^\gamma \sum_{i \in I} a_{j_i^n} (|G_i^n|/|F_i|)^p + \epsilon_n^\beta \sum_{i \in I} \rho_i^q ,
\, \forall \, n \in \mathbb{N}. \notag 
\end{align}
where, in the above, each summand which includes a term of the form \(a_{j_i^n}\), with \(G_i^n = \emptyset\),
is interpreted to be equal to \(0\).
\end{Lem} 
\begin{proof}
We first verify the validity of \eqref{E1}. We set
\[I_0 = \biggl \{ i \in I : \, \sum_{j \in F_i : \, a_j \geq \epsilon_1} a_j /|F_i| 
\geq \rho_i^{q-1} \biggr \}.\]
Suppose that \(I_0 \ne \emptyset \). Then, \(|G_i^0| / |F_i | \geq \rho_i^{q-1}\), for all
\(i \in I_0\) and so 
\[\rho_i \leq (|G_i^0| / |F_i |)^{\frac{1}{q-1}}, \, \forall \, i \in I_0.\]
Note also that
\[ \sum_{j \in F_i : \, a_j \geq \epsilon_1} a_j \leq (\max_{j \in F_i : \, a_j \geq \epsilon_1} a_j ) |G_i^0|  
\leq (1/\epsilon_1) a_{j_i^0} |G_i^0|, \, \forall \, i \in I.\]
We infer from the above that
\begin{align}
\sum_{i \in I_0} \rho_i \sum_{j \in F_i : \, a_j \geq \epsilon_1} a_j /|F_i| &\leq
\sum_{i \in I_0} (|G_i^0| / |F_i |)^{\frac{1}{q-1}} (1/\epsilon_1) a_{j_i^0} (|G_i^0| / |F_i|) \notag \\
&=(1/\epsilon_1) \sum_{i \in I_0}  a_{j_i^0} (|G_i^0| / |F_i |)^p . \notag
\end{align}
The preceding inequality holds trivially when \(I_0 = \emptyset\).
By combining the above relations we obtain
\begin{align}
\sum_{i \in I} \rho_i \sum_{j \in F_i : \, a_j \geq \epsilon_1} a_j /|F_i| &=
\sum_{i \in I_0} \rho_i \sum_{j \in F_i : \, a_j \geq \epsilon_1} a_j /|F_i| +
\sum_{i \in I \setminus I_0} \rho_i \sum_{j \in F_i : \, a_j \geq \epsilon_1} a_j /|F_i| \notag \\
&\leq (1/\epsilon_1) \sum_{i \in I_0}  a_{j_i^0} (|G_i^0| / |F_i |)^p +
\sum_{i \in I \setminus I_0} \rho_i \rho_i^{q-1} \notag \\
&\leq (1/\epsilon_1) \sum_{i \in I}  a_{j_i^0} (|G_i^0| / |F_i |)^p +
\sum_{i \in I} \rho_i^q \notag 
\end{align}
and so \eqref{E1} holds. 

We next prove \eqref{E2}. Fix \(n \in \mathbb{N}\) and set
\[ \rho_{ni} = \sum_{j \in F_i : \, \epsilon_{n+1} \leq a_j < \epsilon_n} a_j /|F_i|, \,
\forall \, i \in I.\]
Observe that
\begin{align}
\rho_{ni} &\leq (\max_{j \in F_i : \, \epsilon_{n+1} \leq a_j < \epsilon_n} a_j )(|G_i^n| / |F_i|)
\notag \\
&\leq \epsilon_n (|G_i^n| / |F_i|) \leq (\epsilon_n / \epsilon_{n+1}) a_{j_i^n}(|G_i^n| / |F_i|),
\forall \, i \in I. \notag 
\end{align}
Define \(I_n = \{i \in I : \, \rho_{ni} \geq \epsilon_n^{\beta} \rho_i^{q-1} \}\).
Suppose that \(I_n \ne \emptyset \) and let \(i \in I_n\).
Then, \(\epsilon_n^{\beta} \rho_i^{q-1} \leq \epsilon_n (|G_i^n| / |F_i|)\). Hence,
\[\rho_i \leq (|G_i^n| / |F_i |)^{\frac{1}{q-1}} 
\epsilon_n^{\frac{1- \beta }{q-1}},
 \, \forall \, i \in I_n.\]
We thus obtain from the preceding relations that
\begin{align}
\sum_{i \in I_n} \rho_i \rho_{ni} &\leq \sum_{i \in I_n} 
(|G_i^n| / |F_i |)^{\frac{1}{q-1}} 
\epsilon_n^{\frac{1- \beta }{q-1}}
(\epsilon_n / \epsilon_{n+1}) a_{j_i^n}(|G_i^n| / |F_i|) \notag \\
&= (\epsilon_n^{\frac{q- \beta}{q-1}} / \epsilon_{n+1}) \sum_{i \in I_n} 
a_{j_i^n}(|G_i^n| / |F_i|)^{1 + \frac{1}{q-1}} =
\epsilon_n^\gamma  \sum_{i \in I_n} 
a_{j_i^n}(|G_i^n| / |F_i|)^p , \notag
\end{align} 
where, in the above, we made use of the fact that \(\epsilon_{n + 1} = \epsilon_n^\alpha\), and
so 
\[\frac{q- \beta}{q-1} - \alpha = \frac{q}{q-1} - \frac{\beta}{q-1} - \alpha =
p -  \frac{\beta}{q-1} - \alpha = \gamma,\]
by the choice of \(\alpha\), \(\beta\), \(\gamma\).
Note also that the preceding estimate is still valid if \(I_n = \emptyset\).
We finally have the estimate
\begin{align}
\sum_{i \in I} \rho_i \sum_{j \in F_i : \, a_j < \epsilon_n} a_j /|F_i| &=
\sum_{i \in I_n} \rho_i \rho_{ni} + \sum_{i \in I \setminus I_n} \rho_i \rho_{ni} + 
\sum_{i \in I} \rho_i \sum_{j \in F_i : \, a_j < \epsilon_{n + 1}} a_j /|F_i| \notag \\
&\leq \sum_{i \in I} \rho_i \sum_{j \in F_i : \, a_j < \epsilon_{n + 1}} a_j /|F_i| +
\epsilon_n^\beta \sum_{i \in I \setminus I_n } \rho_i \rho_i^{q-1} \, + \notag \\
&+ \epsilon_n^\gamma  \sum_{i \in I_n} 
a_{j_i^n}(|G_i^n| / |F_i|)^p \notag \\
&\leq \sum_{i \in I} \rho_i \sum_{j \in F_i : \, a_j < \epsilon_{n + 1}} a_j /|F_i| + \notag \\
&+ \epsilon_n^\gamma \sum_{i \in I} a_{j_i^n} (|G_i^n|/|F_i|)^p + \epsilon_n^\beta \sum_{i \in I} \rho_i^q . \notag
\end{align}
This proves \eqref{E2}.
\end{proof}
\begin{Lem} \label{L5}
Let \(u = \sum_n a_n e_n \) be a finitely supported vector in \(E_p\), with \(\|u \| \leq 1\)
and \(a_n \geq 0\), for all \(n \in \mathbb{N}\). 
We define, for every \(i \in \mathbb{N}\),
\[\lambda_i^0(u) = \sum_{j \in F_i : \, a_j \geq \epsilon_1} a_j /|F_i|, \, \text{ and}, \,
\lambda_i^n(u) = \sum_{j \in F_i : \, a_j < \epsilon_n} a_j /|F_i|, \, \forall \, n \in \mathbb{N}.\]
Let \(I\) be a finite subset of \(\mathbb{N}\) and \((\rho_i)_{i \in I} \subset [0, \infty )\).
Then, the following estimates hold:
\begin{align}
\sum_{i \in I} \rho_i \lambda_i^0(u) &\leq (2/\epsilon_1)\|u\| + \sum_{i \in I} \rho_i^q \notag \\
\sum_{i \in I} \rho_i \lambda_i^n(u) &\leq  \sum_{i \in I} \rho_i \lambda_i^{n+1}(u) +
(n+2) \epsilon_n^\gamma \|u\| + \epsilon_n^\beta \sum_{i \in I} \rho_i^q, \, \forall \, n \in \mathbb{N}. \notag
\end{align}
\end{Lem}
\begin{proof}
Adhering to the notation given in Lemma \ref{L4}, we deduce from Lemma \ref{L3}, that
\[\sum_{i \in I} a_{j_i^n} (|G_i^n|/|F_i|)^p \leq (n+2) \|u\|, \, \forall \, n \in \mathbb{N} \cup \{0\}, \]
as \(a_{j_i^n} \geq \epsilon_{n+1}\), for all \(n \in \mathbb{N} \cup \{0\}\). 
The desired estimates follow now from \eqref{E1} and \eqref{E2}, respectively.
\end{proof}
\begin{Cor} \label{C2}
Let \(B= 4( 1 + 2/\epsilon_1 + \sum_{n=1}^\infty \epsilon_n^\beta + 
\sum_{n=1}^\infty (n +2 )\epsilon_n^\gamma ) \). 
Then, for every \(I \subset \mathbb{N}\), finite, and all choices of scalars
\((\rho_i)_{i \in I}\) with \(\sum_{i \in I} |\rho_i|^q \leq 1\), we have that
\[\biggl \| \sum_{i \in I} \rho_i u_i^* \biggr \| \leq B .\]
\end{Cor}
\begin{proof}
Assume first that all the \(\rho_i\)'s are non-negative. 
Let \(u = \sum_n a_n e_n \) be a finitely supported vector in \(E_p\), with \(\|u \| \leq 1\)
and \(a_n \geq 0\), for all \(n \in \mathbb{N}\). 
Set \(\delta_n = \epsilon_n^\beta + (n +2 )\epsilon_n^\gamma  \), for all \(n \in \mathbb{N}\). 
We now have the following estimates
\begin{align}
\sum_{i \in I} \rho_i u_i^*(u) &= \sum_{i \in I} \rho_i \lambda_i^0(u) + 
\sum_{i \in I} \rho_i \lambda_i^1(u) \notag \\
&\leq 1 + (2/\epsilon_1) + \sum_{i \in I} \rho_i \lambda_i^2(u) + \delta_1, \notag
\end{align}
by applying Lemma \ref{L5}.
Recursive applications of Lemma \ref{L5} now yield that
\[\sum_{i \in I} \rho_i u_i^*(u) \leq 1 + (2/\epsilon_1) + \sum_{i \in I} \rho_i \lambda_i^{n+1}(u) 
+ \sum_{k=1}^n \delta_k, \, \forall \, n \in \mathbb{N}.\]
But since \(I\) is finite and \(\lim_n \epsilon_n = 0\), there exists some \(d \in \mathbb{N}\)
such that \(\lambda_i^{d+1}(u) = 0\) for all \(i \in I\). This in turn implies that
\[\sum_{i \in I} \rho_i u_i^*(u) \leq 1 + (2/\epsilon_1) +  \sum_{k=1}^d \delta_k \leq B /4.\]
Since \((e_n)\) is suppression \(1\)-unconditional, this gives us that
\[\biggl \| \sum_{i \in I} \rho_i u_i^* \biggr \| \leq B /2 .\]
The general case follows immediately from the preceding estimate. 
\end{proof}   
\begin{proof}[Proof of Theorem \ref{T3}.]
Lemma \ref{L2} and Corollary \ref{C2} yield positive constants \(A < B\)
satisfying
\[A \biggl (\sum_{i=1}^n |\rho_i|^q \biggr )^{1/q} \leq
\biggl \| \sum_{i=1}^n \rho_i u_i^* \biggr \| \leq B \biggl (\sum_{i=1}^n |\rho_i|^q \biggr )^{1/q} ,\]
for every \(n \in \mathbb{N}\) and all choices of scalars \((\rho_i)_{i=1}^n\).
Thus, \((u_n^*)\) is equivalent to the \(\ell_q\) basis.
\end{proof}
\begin{Cor} \label{C3}
\(\ell_p\) is isomorphic to a quotient of \(E_p\) and the map
\(Q \colon E_p \to \ell_p \) given by
\[Q \biggl ( \sum_n a_n e_n \biggr ) = \biggl ( \sum_{k \in F_n} a_k /|F_n| \biggr )_{n=1}^\infty \]
is a well-defined, bounded linear surjection.
\end{Cor}
\begin{proof}
Theorem \ref{T3} yields that \(Q\) is a well-defined, bounded, linear operator.
Let \((z_n)\) be the unit vector basis of \(\ell_p\). Then, it is easy to see that
\(Q^*(z_n^*) = u_n^*\), for all \(n \in \mathbb{N}\).
Theorem \ref{T3} now implies that \(Q^*\) is an isomorphic embedding of \(\ell_p^*\) into
\(E_p^*\) and therefore, \(Q\) is a surjection.
\end{proof}
\section{\(E_p\) does not embed in \(C(\alpha)\), \(\alpha < \omega_1\)} \label{S3}
In this section we give the proof of Theorem \ref{T2}. We then apply this result to 
show that for all \(p \geq 1\), \(E_p\) is not isomorphic to a subspace of a \(C(K)\) space
for any countable and compact metric space \(K\), provided that \(\lim_n |F_n| = \infty\).
We first introduce some notation.

{\bf Notation}. Let \((e_n)\) be a normalized basic sequence in some Banach space \(X\). Let 
\(x\) and \(y\) be blocks of \((e_n)\). 
We write \(y \preceq x\) if
\(|e_n^*(y)| \leq |e_n^*(x)|\) for all \(n  \in \mathbb{N}\).
In the case where \(X = C(K)\) and \((e_n)\) consists of non-negative functions on \(K\), we write
\(y \preceq_+ x \) if \(0 \leq e_n^*(y) \leq e_n^*(x)\) for all \(n \in \mathbb{N}\). Evidently,
\(0 \leq y \leq x\), pointwise on \(K\).

It will be more convenient for us to work with non-negative, continuous functions
on some compact metric space. In this setting, Theorem \ref{T2} is reformulated as follows:
\begin{Thm} \label{T4}
Let \(K\) be a compact metric space and \((f_n)\) be a normalized basic sequence in \(C(K)\)
consisting of non-negative functions on \(K\). Assume that there exists a semi-normalized weakly null
block basis \((g_n)\) of \((f_n)\) with \(f_n^*(g_n) \geq 0\) for all \(n \in \mathbb{N}\), satisfying the following
property: For every \(M \in [\mathbb{N}]\), every \(\alpha < \omega_1\) and every \(\epsilon > 0\),
there exists a block basis \((u_m)_{m \in M}\) of \((f_n)\) with finitely many non-zero terms so that
\begin{enumerate}
\item \(u_m \preceq_+ g_m, \, \forall \, m \in M\).
\item \(\| \sum_{m \in M} u_m \| = 1\).
\item \(\| \sum_{m \in F} u_m  \| < \epsilon \), for all \(F \subset M\) with \(F \in S_\alpha\).
\end{enumerate}
Then, \(K\) is uncountable.
\end{Thm}
Assuming that Theorem \ref{T4} is proved, we shall give the proof of Theorem \ref{T2}.
\begin{proof}[Proof of Theorem \ref{T2}.]
Let \(K\) be a compact metric space such that \(C(K)\) contains an isomorph of \(X\).
We show that \(K\) is uncountable. We first choose a normalized basic sequence \((h_n)\) in \(C(K)\),
equivalent to \((e_n)\), and let \(Y\) denote the closed subspace of \(C(K)\) spanned by \((h_n)\).
Let \(T \colon X \to Y\) be an isomorphism with \(Te_n = h_n\) for all \(n \in \mathbb{N}\).
Set \(c = \|T^{-1} \|\). Let \((x_n)\) be a semi-normalized block basis
of \((e_n)\) according to the hypothesis of the theorem, and set
\(y_n = c Tx_n\), for all \(n \in \mathbb{N}\). Then, \((y_n)\) is a semi-normalized block basis
of \((h_n)\). It is a routine to check that \((y_n)\) satisfies the same property, relatively to
\((h_n)\), that \((x_n)\) enjoys in the statement of the theorem.

We next set \(f_n = |h_n|\), for all \(n \in \mathbb{N}\). Clearly, \((f_n)\) is normalized 
weakly null in \(C(K)\). By passing to a subsequence if necessary, we may assume, without loss
of generality, that \((f_n)\) is basic. Write \(y_n = \sum_{i \in I_n} b_i h_i\), where
\(I_1 < I_2 < \dots \) is a sequence of successive finite subsets of \(\mathbb{N}\) and
\((b_i)_{i \in I_n}\) are scalars for all \(n \in \mathbb{N}\). Put \(g_n = \sum_{i \in I_n} |b_i|f_i\)
for all \(n \in \mathbb{N}\). The unconditionality of \((h_n)\) yields that \((g_n)\) is
semi-normalized. Indeed, letting \(C > 0\) denote the unconditional constant of \((h_n)\), fix some
\(t \in K\). For each \(n \in \mathbb{N}\), choose a collection of signs \((\sigma_i)_{i \in I_n}\)
so that \( g_n(t) = \sum_{i \in I_n} b_i \sigma_i h_i(t)\). Then,
\[g_n(t) \leq \biggl \| \sum_{i \in I_n } b_i \sigma_i h_i \biggr \| \leq C  
\biggl \| \sum_{i \in I_n } b_i h_i \biggr \| = C \|y_n\|.\] 
Hence, as \(t \in K\) was arbitrary,
\(\|y_n\| \leq \|g_n \| \leq C \|y_n\|\), for all \(n \in \mathbb{N}\). Since \((y_n)\) is
semi-normalized, so is \((g_n)\).

We also obtain that \((g_n)\) is weakly null. Indeed, let \(t \in K\) and write, as we did in the previous paragraph,
\( g_n(t) = \sum_{i \in I_n} b_i \sigma_i h_i(t)\) for all \(n \in \mathbb{N}\). 
Set \(r_n = \sum_{i \in I_n } b_i \sigma_i h_i \) for all \(n \in \mathbb{N}\). Then, \((r_n)\)
is a semi-normalized block basis of \((h_n)\). Since the latter is shrinking, \((r_n)\) is weakly null
and so we get that
\(\lim_n r_n(t) = 0\). Thus, \(\lim_n g_n(t) = 0\) as well. The assertion follows since \(t \in K\) was
arbitrary and \((g_n)\) is semi-normalized.

We finally show that \((g_n)\) satisfies the property described in the statement of Theorem \ref{T4}.
To this end, let \(M \in [\mathbb{N}]\), \(\alpha < \omega_1\) and \(\epsilon > 0\). We may choose 
a block basis \((v_m)_{m \in M}\) of \((h_n)\) with finitely many non-zero terms so that
\begin{align}
&v_m \preceq y_m, \, \forall \, m \in M, \text{ and }  \notag \\
&\biggl \| \sum_{m \in M} v_m \biggr \| =1, \text{ while } \biggl \| \sum_{m \in F} v_m \biggr \| < \epsilon /C,
\, \forall \, F \subset M, \, F \in S_\alpha. \label{E3}
\end{align}
Write \(v_m = \sum_{i \in I_m} \lambda_i h_i \), where \(|\lambda_i | \leq |b_i|\)
for all \(i \in I_m\) and \(m \in M\). Define \(w_m = \sum_{i \in I_m} |\lambda_i| f_i\).
It is clear that \(w_m \preceq_+ g_m\), for all \(m \in M\).
\eqref{E3} now implies that
\[\biggl \| \sum_{m \in M} w_m \biggr \| \geq \biggl \| \sum_{m \in M} v_m \biggr \| =1.\]
Using an argument similar to that in the previous paragraph, based on the unconditionality of \((h_n)\),
we obtain that
\[\biggl \| \sum_{m \in F} w_m \biggr \| \leq C \biggl \| \sum_{m \in F} v_m \biggr \|, \,
\forall \, F \subset M.\]    
Hence, we deduce from \eqref{E3}, that
\(\| \sum_{m \in F} w_m \| < \epsilon\), for all \(F \subset M\) with \(F \in S_\alpha\).

Setting \(u_m = w_m /\|\sum_{m \in M} w_m \|\), for all \(m \in M\), we see that
\((u_m)\) satisfies conditions \((1)\)-\((3)\) in the hypothesis of Theorem \ref{T4}.
Therefore, \(K\) is uncountable.
\end{proof}
The proof of Theorem \ref{T4} requires some preparatory steps.
\begin{Lem} \label{L6}
Let \(K\) be a set, \(\rho > 0 \) and \((\epsilon_n)\) a sequence of positive scalars.
Let \((g_n)\) be a sequence of non-negative functions on \(K\), converging pointwise to
the zero function on \(K\). Assume that for every \(M \in [\mathbb{N}]\), \(M = (m_n)\),
there exist \(t \in K\), \(I \in [\mathbb{N}]\), \(I=(i_n)\), and a sequence \((u_{m_n})\)
of non-negative functions on \(K\), satisfying the following:
\begin{enumerate}
\item \(u_{m_n} \leq g_{m_n}\), for all \(n \in \mathbb{N}\).
\item \(\sum_{i=1}^{i_1} u_{m_i}(t) \geq \rho\).
\item \((1 + \epsilon_{m_{i_n}}) \sum_{i=1}^{i_n} u_{m_i}(t) \leq
 \sum_{i=1}^{i_{n+1}} u_{m_i}(t) \), for all \(n \in \mathbb{N}\).
\end{enumerate}
Then, \(K\) is uncountable.
\end{Lem}
\begin{proof}
Suppose to the contrary, that \(K\) is countable and let
\((t_n)\) be an enumeration of the elements of \(K\). Choose \(m_1 \in \mathbb{N}\)
arbitrarily. Using the fact that \((g_n)\) is pointwise null on \(K\), we can inductively
choose a nested sequence \(M_1 \supset M_2 \supset \dots \) of infinite subsets of \(\mathbb{N}\)
so that letting \(m_{k+1} = \min M_k \), for all \(k \in \mathbb{N}\), the following hold:
\begin{align}
m_k &< m_{k+1}, \, \forall \, k \in \mathbb{N}, \text{ and}, \notag \\
\sum_{m \in M_k} g_m(t_j) &< \rho \epsilon_{m_k}, \, \forall \, j \leq k,
\, \forall \, k \in \mathbb{N}. \label{E4}
\end{align}
Set \(M = (m_k)\). Our assumptions yield a sequence \((u_{m_k})\) of non-negative functions
on \(K\), an infinite sequence of positive integers \((i_k)\), and a \(t \in K\), satisfying
\((1)\)-\((3)\). Suppose \(t = t_d\) for some \(d \in \mathbb{N}\). We infer from \((3)\) that
\[ (1 + \epsilon_{m_{i_d}}) \sum_{i=1}^{i_d} u_{m_i}(t_d) \leq
 \sum_{i=1}^{i_{d+1}} u_{m_i}(t_d). \] 
Note that \(m_{i_{d+1}} \geq m_{i_d + 1} = \min M_{i_d}\), and so
\((m_i)_{i=i_d + 1}^{i_{d+1}} \subset M_{i_d}\). It follows now, as
\(u_{m_j} \leq g_{m_j}\) for all \(j \in \mathbb{N}\), that
\[\sum_{i = i_d + 1}^{i_{d+1}} u_{m_i}(t_d) \leq \sum_{i = i_d + 1}^{i_{d+1}} g_{m_i}(t_d)
\leq \sum_{m \in M_{i_d}} g_m(t_d).\]
\eqref{E4} now implies that 
\(\sum_{i = i_d + 1}^{i_{d+1}} u_{m_i}(t_d) < \rho \epsilon_{m_{i_d}}\), as
\(d \leq i_d\). We deduce from the above that
\[ (1 + \epsilon_{m_{i_d}}) \sum_{i=1}^{i_d} u_{m_i}(t_d) < 
\sum_{i=1}^{i_d} u_{m_i}(t_d) + \rho \epsilon_{m_{i_d}}, \] 
whence,
\[ \sum_{i=1}^{i_1} u_{m_i}(t_d) \leq  
\sum_{i=1}^{i_d} u_{m_i}(t_d) < \rho,\]
contradicting \((2)\). Thus, \(K\) is uncountable.
\end{proof}
We shall next attempt to localize the content of the previous lemma, in terms
of countable ordinals.
Fix a compact metric space \(K\), a normalized basic sequence \((f_n)\) in \(C(K)\),
consisting of non-negative functions on \(K\),
and a semi-normalized block basis \((g_n)\) of \((f_n)\) with \(f_n^*(g_n) \geq 0\)
for all \(n \in \mathbb{N}\). Assume that \((g_n)\) is weakly null. 
We also fix \(\rho \in (0,1)\) and a null sequence of positive scalars \((\epsilon_n)\).
\begin{Def}
Let \(M \in [\mathbb{N}]\), \(M= (m_n)\), and \(F = \{m_{i_1}, < \dots < m_{i_n}\}\),
be a finite subset of \(M\).
An \(F\)-chain supported by \(M\), is a finite sequence \((u_{m_i})_{i=1}^{i_n} \subset C(K)\)
with \(u_{m_i} \preceq_+ g_{m_i}\), for all \(i \leq i_n\), for which 
there is some \(t \in K\) satisfying the following:
\begin{enumerate}
\item \(\sum_{i=1}^{i_1} u_{m_i}(t) \geq \rho\).
\item \((1 + \epsilon_{m_{i_k}}) \sum_{i=1}^{i_k} u_{m_i}(t) \leq
 \sum_{i=1}^{i_{k+1}} u_{m_i}(t) \), for all \( k < n\).
\end{enumerate}
\end{Def}
Given \(M \in [\mathbb{N}]\), we set
\[\mathcal{F}_M = \{ F \in [M]^{< \infty}: \, \text{ there exists an } F-\text{ chain supported by } M \} 
\cup \{\emptyset\}.\]
Evidently, \(\mathcal{F}_M\) is a hereditary family of finite subsets of \(M\). 
If \(\rho \leq \inf_n \|g_n\|\), then it is easy to see that \(\mathcal{F}_M\)
contains the singletons of \(M\). We also have that
\(\mathcal{F}_L \subset \mathcal{F}_M\), for all \(L \in [M]\).
\begin{Lem} \label{L7}
Assume that for every \(M \in [\mathbb{N}]\),
\(\mathcal{F}_M\) is not compact in the topology of pointwise convergence.
Then, \(K\) is uncountable.
\end{Lem}
\begin{proof}
Let \(M \in [\mathbb{N}]\), \(M= (m_n)\). Since \(\mathcal{F}_M\) is hereditary and not pointwise
compact, there exists an infinite sequence of positive integers \((i_n)\), so that
\(\{m_{i_1}, \dots , m_{i_n} \} \in \mathcal{F}_M\), for all \(n \in \mathbb{N}\).
Set \(M_n = \{m \in M : m \leq m_{i_n} \}\), for all \(n \in \mathbb{N}\). We can select,
for each \(n \in \mathbb{N}\), a sequence \((u_m^n)_{m \in M_n} \subset C(K)\)
with \(u_m^n \preceq_+ g_m\), for all \(m \in M_n\), and a \(t_n \in K\),
so that
\begin{align}
&\sum_{m \in M_1} u_m^n (t_n) \geq \rho, \text{ and}, \notag \\
&(1 + \epsilon_{m_{i_k}}) \sum_{m \in M_k} u_m^n (t_n) \leq \sum_{m \in M_{k+1}} u_m^n (t_n),
\, \forall \, k < n. \notag 
\end{align}
Let \(m \in M\). Since \(u_m^n \preceq_+ g_m\), for all \(n \in \mathbb{N}\) with
\(m \in M_n\), we have that \((u_m^n)_{n : \, m \in M_n}\) is a bounded sequence in the finite dimensional
subspace of \(C(K)\) spanned by \((f_i)_{i \in \mathrm{ supp } \, g_m }\)
(the support of \(g_m\) is measured with respect to \((f_n)\)).
We can thus inductively
select a nested sequence \(J_1 \supset J_2 \supset \dots\) of infinite subsets of \(\mathbb{N}\)
so that, for all \(k \in \mathbb{N}\), the sequence \((u_{m_k}^n)_{n \in J_k}\) converges
uniformly on \(K\) to a function \(u_{m_k} \preceq_+ g_{m_k}\).

We next choose integers \(j_1 < j_2 < \dots \) with \(j_k \in J_k\) for all \(k \in \mathbb{N}\).
We may assume, without loss of generality thanks to the compactness of \(K\), that
\(\lim_n t_{j_n} = t \in K\). It follows that \(\lim_n u_m^{j_n} = u_m\), uniformly on \(K\),
for all \(m \in M\). We deduce from the choices made above, that
\begin{align}
&\sum_{m \in M_1} u_m^{j_n}(t_{j_n}) \geq \rho, \text{ and}, \label{E5} \\
&(1 + \epsilon_{m_{i_k}}) \sum_{m \in M_k} u_m^{j_n} (t_{j_n}) \leq \sum_{m \in M_{k+1}} u_m^{j_n} (t_{j_n}),
\, \forall \, k < n. \label{E6} 
\end{align}
Finally, let \(n \to \infty\) in \eqref{E5}, while fix \(k \in \mathbb{N}\) and let \(n \to \infty\) in
\eqref{E6}, to conclude that
\begin{align}
&\sum_{m \in M_1} u_m(t) \geq \rho, \text{ and}, \notag \\
&(1 + \epsilon_{m_{i_k}}) \sum_{m \in M_k} u_m(t) \leq \sum_{m \in M_{k+1}} u_m(t),
\, \forall \, k \in \mathbb{N}. \notag 
\end{align}
Since \((g_n)\) is weakly null, we see that the hypotheses of Lemma \ref{L6} are fulfilled and hence
\(K\) is uncountable.
\end{proof}
The key step in proving Theorem \ref{T4} is contained in our next proposition 
\begin{Prop} \label{P2}
Assume that for every \( M \in [\mathbb{N}]\), \(M = (m_n)\), every \(\alpha < \omega_1\)
and every \(\epsilon > 0\), there exist \(n \in \mathbb{N}\) and \((u_{m_i})_{i=1}^n \subset C(K)\)
with \(u_{m_i} \preceq_+ g_{m_i}\) for all \(i \leq n\), so that
\[\biggl \| \sum_{i=1}^n u_{m_i} \biggr \| = 1, \text{ yet}, \, 
\biggl \| \sum_{m \in F} u_m \biggr \| < \epsilon, \, \forall \,
F \subset \{m_1, \dots m_n \}, \, F \in S_\alpha.\]
Then, \(\mathcal{F}_M\) is pointwise compact for no \(M \in [\mathbb{N}]\).
\end{Prop} 
\begin{proof}
Let us suppose that for some \(L \in [\mathbb{N}]\), \(\mathcal{F}_L\)
is pointwise compact. By a classical result \cite{ms},
\(\mathcal{F}_L\) will then be homeomorphic to some ordinal interval
\([1, \beta]\), with \(\beta < \omega_1\). The dichotomy result of \cite{g}
now yields \(M \in [L]\) and \(\alpha < \omega_1\) so  
that \(\mathcal{F}_L \cap [M]^{< \infty} \subset S_\alpha\). Therefore,
\(\mathcal{F}_M \subset S_\alpha\), for some \(M \in [L]\) and \(\alpha < \omega_1\).

Since \(\lim_n \epsilon_n = 0\) and \(0 < \rho < 1\), we can assume, without loss of generality, that
\[\rho \prod_{i=1}^\infty (1 + \epsilon_{m_i}) < 1, \text{ where}, \, M =(m_i).\]
We next choose \(0 < \epsilon < \rho \) with
\((\rho + \epsilon ) \prod_{i=1}^\infty (1 + \epsilon_{m_i}) < 1\).
Our assumptions yield \(n \in \mathbb{N}\) and \((u_{m_i})_{i=1}^n \subset C(K)\)  
with \(u_{m_i} \preceq_+ g_{m_i}\) for all \(i \leq n\), so that
\[\biggl \| \sum_{i=1}^n u_{m_i} \biggr \| = 1, \text{ yet}, \, 
\biggl \| \sum_{m \in F} u_m \biggr \| < \epsilon, \, \forall \,
F \subset \{m_1, \dots m_n \}, \, F \in S_\alpha.\]
Choose \(t \in K\) with \(\sum_{i=1}^n u_{m_i}(t) = 1\).
Note that \(u_{m_i}(t) \leq \|u_{m_i}\| < \epsilon \) for all \(i \leq n\),
as \(S_\alpha\) contains all singletons.

Let \(i_1\) be the smallest integer in \(\{1, \dots , n\}\) 
satisfying \(\sum_{i=1}^{i_1} u_{m_i}(t) \geq \rho\).
Clearly, \(i_1 > 1\) as \(u_{m_1}(t) < \epsilon < \rho\).
We also have that \(i_1 < n\) as
\(\sum_{i=1}^{n-1} u_{m_i} (t) \geq 1 - \epsilon > \rho\).

We next let \(i_2\) be the smallest integer in \(\{i_1 + 1, \dots , n\}\)
satisfying

\((1 + \epsilon_{m_{i_1}}) \sum_{i=1}^{i_1} u_{m_i}(t) \leq \sum_{i=1}^{i_2} u_{m_i}(t).\)
Note that \(i_2\) exists since
\begin{align}
(1 + \epsilon_{m_{i_1}}) \sum_{i=1}^{i_1} u_{m_i}(t) &=
(1 + \epsilon_{m_{i_1}}) \biggl ( \sum_{i=1}^{i_1 -1 } u_{m_i}(t) +  u_{m_{i_1}}(t) \biggr ) \notag \\
&< (1 + \epsilon_{m_{i_1}}) (\rho + \epsilon) 
< 1, \notag 
\end{align} 
by the minimality of \(i_1\) and since 
\(\|u_{m_{i_1}} \| < \epsilon\).
  
We continue in this manner selecting integers \(i_2 < \dots < i_k \leq n\) so that
for each \(l \in \{2, \dots k\}\), \(i_l\) is the smallest integer in \(\{i_{l-1} + 1, \dots n \}\)
satisfying
\[(1 + \epsilon_{m_{i_{l-1}}}) \sum_{i=1}^{i_{l-1}} u_{m_i}(t) \leq 
\sum_{i=1}^{i_l} u_{m_i}(t).\]
Moreover, \(i_k\) satisfies
\(1 \leq (1 + \epsilon_{m_{i_k}}) \sum_{i=1}^{i_k} u_{m_i}(t)\).

It is clear that this process can not carry on after the \(k\)-th stage.
We now obtain the following estimates:
\begin{align}
1 &\leq (1 + \epsilon_{m_{i_k}}) \sum_{i=1}^{i_k} u_{m_i}(t)=
(1 + \epsilon_{m_{i_k}}) \biggl ( \sum_{i=1}^{i_k -1 } u_{m_i}(t) +  u_{m_{i_k}}(t) \biggr ) 
\notag \\
&\leq (1 + \epsilon_{m_{i_k}}) \biggl ( (1 + \epsilon_{m_{i_{k-1}}}) \sum_{i=1}^{i_{k-1}} u_{m_i}(t)
+ u_{m_{i_k}}(t) \biggr ), 
\text{ since } i_k \text{ is minimal}, \notag \\
&= \biggl (\prod_{j= k-1}^k (1 + \epsilon_{m_{i_j}}) \biggr ) \sum_{j=1}^{i_{k-1}} u_{m_j}(t) +
\biggl ( \prod_{j= k}^k (1 + \epsilon_{m_{i_j}}) \biggr )u_{m_{i_k}}(t). \notag
\end{align}
A similar argument, using the minimality of \(i_{k-1}\), yields that
\begin{align}
1 &\leq \biggl ( \prod_{j= k-2}^k (1 + \epsilon_{m_{i_j}}) \biggr ) \sum_{j=1}^{i_{k-2}} u_{m_j}(t) +
\notag \\
&+ \biggl ( \prod_{j= k-1}^k (1 + \epsilon_{m_{i_j}}) \biggr ) u_{m_{i_{k-1}}}(t) +
\biggl ( \prod_{j= k}^k (1 + \epsilon_{m_{i_j}}) \biggr ) u_{m_{i_k}}(t). \notag
\end{align}
We continue in this fashion, using the minimality of \(i_l\) for \(2 \leq l \leq k\) and after
\(k-1\) steps, we reach the estimate
\begin{align}
1 &\leq \biggl ( \prod_{j=1}^k (1 + \epsilon_{m_{i_j}}) \biggr ) \sum_{j=1}^{i_1} u_{m_j}(t) +
\sum_{j=2}^k 
\biggl ( \prod_{l=j}^k (1 + \epsilon_{m_{i_l}}) \biggr ) u_{m_{i_j}}(t) \notag \\
&\leq \biggl ( \prod_{j=1}^k (1 + \epsilon_{m_{i_j}}) \biggr ) \sum_{j=1}^{i_1 - 1} u_{m_j}(t) +
\sum_{j=1}^k \biggl ( \prod_{l=j}^k (1 + \epsilon_{m_{i_l}}) \biggr ) u_{m_{i_j}}(t) \notag \\
&< \rho \prod_{j=1}^k (1 + \epsilon_{m_{i_j}}) +
\biggl (\prod_{j=1}^k (1 + \epsilon_{m_{i_j}}) \biggr ) \sum_{j=1}^k u_{m_{i_j}}(t), \notag \\
&\text{by the minimality of } i_1. \notag
\end{align}
We infer from the above that
\[ 1 < \biggl (\prod_{j=1}^k (1 + \epsilon_{m_{i_j}}) \biggr ) 
\biggl ( \rho + \sum_{j=1}^k u_{m_{i_j}}(t) \biggr ).\] 
However, the choice of the indices \((i_j)_{j=1}^k\) guarantees that
\(\{m_{i_j} : \, j \leq k\} \in \mathcal{F}_M\) and so
\(\{m_{i_j} : \, j \leq k\} \in S_\alpha\). Therefore, by the choice of the
\(u_{m_i}\)'s, we have that 
\[\sum_{j=1}^k u_{m_{i_j}}(t) \leq \biggl \| \sum_{j=1}^k u_{m_{i_j}} \biggr \| < \epsilon\]
and so
\[ 1 < (\rho + \epsilon ) \prod_{j=1}^k (1 + \epsilon_{m_{i_j}}) < 1.\] 
This contradiction yields the assertion of the proposition.
\end{proof}
\begin{proof}[Proof of Theorem \ref{T4}.]
The hypotheses of the theorem and Proposition \ref{P2} yield that
\(\mathcal{F}_M\) is not pointwise compact for every \(M \in [\mathbb{N}]\).
Lemma \ref{L7} now implies that \(K\) is uncountable.
\end{proof}
Our final task is to apply Theorem \ref{T2} to \(E_p\) and show it does
not embed into \(C(\alpha)\) for all \(\alpha < \omega_1\).
\begin{Cor} \label{C5}
Assume that \(\lim_n |F_n| = \infty\).
Then, for every \(p \geq 1 \), \(E_p\) is not isomorphic to a subspace
of a \(C(K)\) space for any countable and compact metric space \(K\).
\end{Cor}
\begin{proof}
Let \((e_n)\) be the natural basis for \(E_p\). We know it is normalized, unconditional
and shrinking. We set \(x_n = \sum_{i \in F_n } e_i \) for all \(n \in \mathbb{N}\).
Then, \((x_n)\) is a normalized block basis of \((e_n)\) and thus, it is weakly null. The assertion of the corollary
will be established, once we show that \((x_n)\) satisfies
the assumptions of Theorem \ref{T2}. To this end, let \(M \in [\mathbb{N}]\), \(\alpha < \omega_1\)
and \(\epsilon \in (0,1)\). We may assume, without loss of generality, that
\[ \min M > 2^{p+1}/ \epsilon, \, \text{ and}, \sum_{m \in M} 1/|F_m| < 1/2^{p+1}.\]
Let \(k = \min F_{\min M}\). It is clear that \(1/ k < \epsilon / 2^{p+1}\).
It is shown in \cite{amt} (cf. also, \cite{ag}),
that there exist finitely supported probability measures
\(\mu_1 < \dots < \mu_k\) on \(\mathbb{N}\) so that for all \(i \leq k\),
\[\mathrm{ supp } \, \mu_i \subset M, \text{ and}, \, \mu_i(F) < \epsilon /2^{p+1}, \, \forall \,
F \in S_\alpha.\]
Fix \(i \leq k\). For each \(m \in \mathrm{ supp } \, \mu_i\) we can find an initial segment
\(G_{im}\) of \(F_m\) such that
\begin{equation} \label{E7}
|G_{im}|/|F_m| \leq \mu_i(m)^{1/p} < (|G_{im}|/|F_m|) + (1/|F_m|). 
\end{equation}
We set \(\tau_i = \sum_{ m \in \mathrm{ supp } \, \mu_i} (|G_{im}|/|F_m|)^p
e_{\max G_{im}}^*\). We next let \(\tau = \sum_{i=1}^k \tau_i \).
Since \(\mu_i\) is a probability measure, \eqref{E7} yields that
\(\tau_i\) is a \(p\)-measure for all \(i \leq k\). We also have that
\((\tau_i)_{i=1}^k\) is admissible, as \(k \leq \min F_m\), for all \(m \in M\).
Hence, \(\tau \in \mathcal{M}_p\).
We now define
\[v_m = (1/k) \sum_{j \in G_{im}} e_j , \, \forall \, m \in \mathrm{ supp } \, \mu_i,
\, \forall \, i \leq k.\]
It is clear that \(v_m \preceq x_m \), for all \(m \in \cup_{i=1}^k \mathrm{ supp } \, \mu_i\).
We now have the following estimate
\begin{align}
\tau \biggl ( \sum_{i=1}^k \sum_{m \in \mathrm{ supp } \, \mu_i} v_m \biggr ) &=
\sum_{i=1}^k \sum_{m \in \mathrm{ supp } \, \mu_i} \tau_i(v_m) \label{E8} \\
&= \sum_{i=1}^k (1/ k) \sum_{m \in \mathrm{ supp } \, \mu_i} (|G_{im}|/|F_m|)^p .
\notag
\end{align}
Taking in account \eqref{E7}, we have that
\[\mu_i(m) < \bigl [(|G_{im}|/|F_m|) + (1/|F_m|)\bigr ]^p \leq 2^p (|G_{im}|/|F_m|)^p + (2^p /|F_m|^p),\] 
for all \(m \in \mathrm{ supp } \, \mu_i \) and \(i \leq k\),
and so,
\[1 = \sum_{ m \in \mathrm{ supp } \, \mu_i} \mu_i(m) \leq  
2^p \sum_{ m \in \mathrm{ supp } \, \mu_i} (|G_{im}|/|F_m|)^p 
+ 2^p \sum_{ m \in \mathrm{ supp } \, \mu_i} 1/|F_m|^p, \] 
for all \(i \leq k\).
\eqref{E8} now implies that
\begin{align}
2^p \biggl \|\sum_{i=1}^k \sum_{m \in \mathrm{ supp } \, \mu_i} v_m \biggr \| &\geq
1 - 2^p \sum_{i=1}^k  (1/k) \sum_{ m \in \mathrm{ supp } \, \mu_i} 1/|F_m|^p \notag \\
&\geq 1 - 2^p \sum_{m \in M} (1/|F_m|) > 1 - 2^p(1/2^{p+1}) = 1/2. \notag
\end{align}
Hence, \(D = \|\sum_{i=1}^k \sum_{m \in \mathrm{ supp } \, \mu_i} v_m \| > 1/2^{p+1}\).
We finally set 
\[u_m = (1/D) v_m, \, \forall \, m \in \cup_{i=1}^k \mathrm{ supp } \, \mu_i.\]
Put also \(u_m = 0\), if \(m \notin \cup_{i=1}^k \mathrm{ supp } \, \mu_i\).
It is clear that \(\| \sum_{m \in M} u_m \| = 1\) and that
\(u_m \preceq x_m\) for all \(m \in M\), as \(k > 2^{p+1}\).  

Finally, let \( F \subset M\), \(F \in S_\alpha\).
We need to show that \(\|\sum_{m \in F} u_m \| < \epsilon\).
It will suffice showing that \(\|\sum_{m \in F} v_m \| < \epsilon /2^{p+1}\).
We can assume that \(F \subset \cup_{i=1}^k \mathrm{ supp } \, \mu_i \).

Given \(m \in F\), let \(i_m\) denote the unique \(i \leq k\) with \(m \in \mathrm{ supp } \, \mu_i \).
It follows that
\[\sum_{m \in F} v_m = \sum_{m \in F} \sum_{j \in G_{i_m m}}(1/k) e_j .\]
Let \(\Lambda \in \mathcal{M}_p\). If \(\Lambda = e_l^*\) for some \(l \in \mathbb{N}\),
then clearly,
\[\Lambda \biggl (\sum_{m \in F} v_m \biggr ) \leq 1/k < \epsilon /2^{p+1}.\]
When \(|\mathrm{ supp } \, \Lambda | > 1\), then \(\Lambda\) is of the form
\[\Lambda = \sum_{i=1}^\infty (|G_i|/|F_i|)^p e_{\max G_i }^* ,\]
where each \(G_i\) is an initial segment of \(F_i\).  
It follows now that if
\(\Lambda(v_m) \ne 0\) then
\(\mathrm{ supp } \, \Lambda \cap F_m \ne \emptyset \) and \(G_m \subset G_{i_m m}\).
Thus,
\[\Lambda (v_m ) \leq (|G_{i_m m}|/|F_m|)^p (1/k), \, \forall \, m \in F\]
and so
\begin{align}
\Lambda \biggl (\sum_{m \in F} v_m \biggr ) &\leq (1/k) \sum_{m \in F}
(|G_{i_m m}|/|F_m|)^p \notag \\
&\leq (1/k) \sum_{m \in F} \mu_{i_m}(m), \text{ by } \eqref{E7} \notag \\
&= (1/k) \sum_{i=1}^k \mu_i(F) < \epsilon/2^{p+1}, \notag
\end{align}
as required.    
\end{proof}

\end{document}